\documentclass[11pt,leqno]{amsart}
\usepackage{amssymb,amsmath,amsthm,amsfonts}
\usepackage{graphicx}
\usepackage{float}
\usepackage{xcolor}
\usepackage{enumitem}
\usepackage[normalem]{ulem}
\usepackage{esint}

\DeclareMathOperator{\diam}{diam}
\DeclareMathOperator{\dist}{dist}

\newcommand{\N}{\mathbb{N}} 
\newcommand{\Z}{\mathbb{Z}}   
\newcommand{\R}{\mathbb{R}}

\theoremstyle{plain}

\newtheorem{Thm}{Theorem}
\newtheorem{theoremA}{Theorem}
\newtheorem{Cor}[Thm]{Corollary}
\newtheorem{Lem}[Thm]{Lemma}

\newtheorem{Prop}[Thm]{Proposition}

\newtheorem{Rem}[Thm]{Remark}

\numberwithin{Def}{section} 
\numberwithin{Thm}{section} \numberwithin{equation}{section}

\title{On the Assouad dimension of  Weierstrass function graphs}

\begin{document}
\author[Efstathios-K. Chrontsios-Garitsis]{Efstathios-K. Chrontsios-Garitsis}
\address{Department of Mathematics \\  The Ohio State University, Columbus, OH}
\email{chrontsios.1@osu.edu, echronts@gmail.com}

\begin{abstract}
The class of Weierstrass functions $W_{a,b}$ is one of the first class of examples of continuous and nowhere differentiable real functions. 
A challenging line of research has been to determine the various dimensions of graphs $G(W_{a,b})$ of such functions. For instance, the Hausdorff dimension of $G(W_{a,b})$ was only recently determined by Shen in 2018 \cite{ShenWeier}, after a long series of partial results by many different authors. While the Assouad dimension of $G(W_{a,b})$ remains an open problem, also posed as a question by Fraser in 2020 \cite{FraserBook}, there have been many indications that it might be equal to $2$. Such indications include the graph of Wiener processes \cite{Howroyd_Yu_Wiener_Process}, the graphs of almost all H\"older functions in the Baire category sense \cite{FengFraser_Holder_typical}, and the graphs of Weierstrass functions after a series of countably many reflections \cite{ChronT_spec_holder} all having Assouad dimension equal to $2$. In this paper we show that this is not the case, providing a quantitative upper bound on the Assouad dimension of $G(W_{a,b})$ that is strictly less than $2$. In particular, we show that such a bound is true for a class of generalized Weierstrass functions $W_{a,b}^\phi(x)
=
\sum_{j=0}^\infty a^j\phi(b^j x)$, which includes $W_{a,b}$ and the class of Takagi functions. The latter fact is used to also answer in the negative a conjecture of Yu \cite{Yu_Takagi} on the Assouad dimension of graphs of Takagi functions for parameters $a\in (0,1)$, $b\in (1/a, \infty)\cap \Z$.
\end{abstract}

\maketitle

\section{Introduction}

For quite some time, mathematicians were under the impression that a continuous  function may be non-differentiable, but only at a relatively small subset of its domain. One of the first examples to disprove this was given by Weierstrass during his 1872 lecture at the Royal Prussian Academy of Science in Berlin (see for instance \cite{duBoisReymond1875_firstWeier,HardyWeier}). In fact, Weierstrass provided a whole class of such continuous nowhere differentiable examples for a range of parameters $a,b$. Given $a\in (0,1)$ and integer $b\in (1/a,\infty)$, the $1$-periodic Weierstrass function $W_{a,b}:\R\to \R$ is defined to be
	$$
	W_{a,b}(x):=\sum_{j=0}^{\infty} a^j \cos (2\pi b^j x), \qquad 0 \le x \le 1.
	$$ 
	Another class of continuous nowhere differentiable functions is that   introduced by Takagi in 1903 \cite{Takagi_first} (see also \cite{Takagi_survey} for a detailed survey). Given $a,b$ as above, the $1$-periodic Takagi function $T_{a,b}:\R\to \R$ is defined to be
	$$
	T_{a,b}(x)=\sum_{j=0}^{\infty}a^j D(b^j x), \qquad 0 \le x \le 1,
	$$ 
	where $D$ is the $1$-periodic sawtooth function with $D(t)=t$ for all $t\in [0,1/2]$ and $D(t)=1-t$ for all $t\in [1/2,1]$. While at first such functions were considered to be a technical pathology, it was later proved by Banach in 1931 that ``most'' continuous functions, in the Baire category sense, are in fact nowhere differentiable \cite{banach1931baire}. This celebrated result brought these classes of continuous nowhere differentiable functions to the center of attention within many areas of math. A few recent examples include the study of these functions within the areas of partial differential equations~\cite{Appl_EllipticPDESobolev, Appl_WaveEquationsRoughCoefficients}, stochastic calculus \cite{Appl_StochasticCalculus}, harmonic analysis \cite{Appl_HarmonicRealAnalysis}, potential theory \cite{Appl_PotentialTheory}, approximation theory \cite{Appl_ApproximationTheory}, probability \cite{RomanowskaRandomizedWeierstrass}, number theory \cite{AJChron_mobius}, and dynamical systems \cite{Appl_DynamicalSystems}.

    \begin{figure}[t]
    \centering
    \includegraphics[height=0.55\textheight, width=0.99\textwidth]{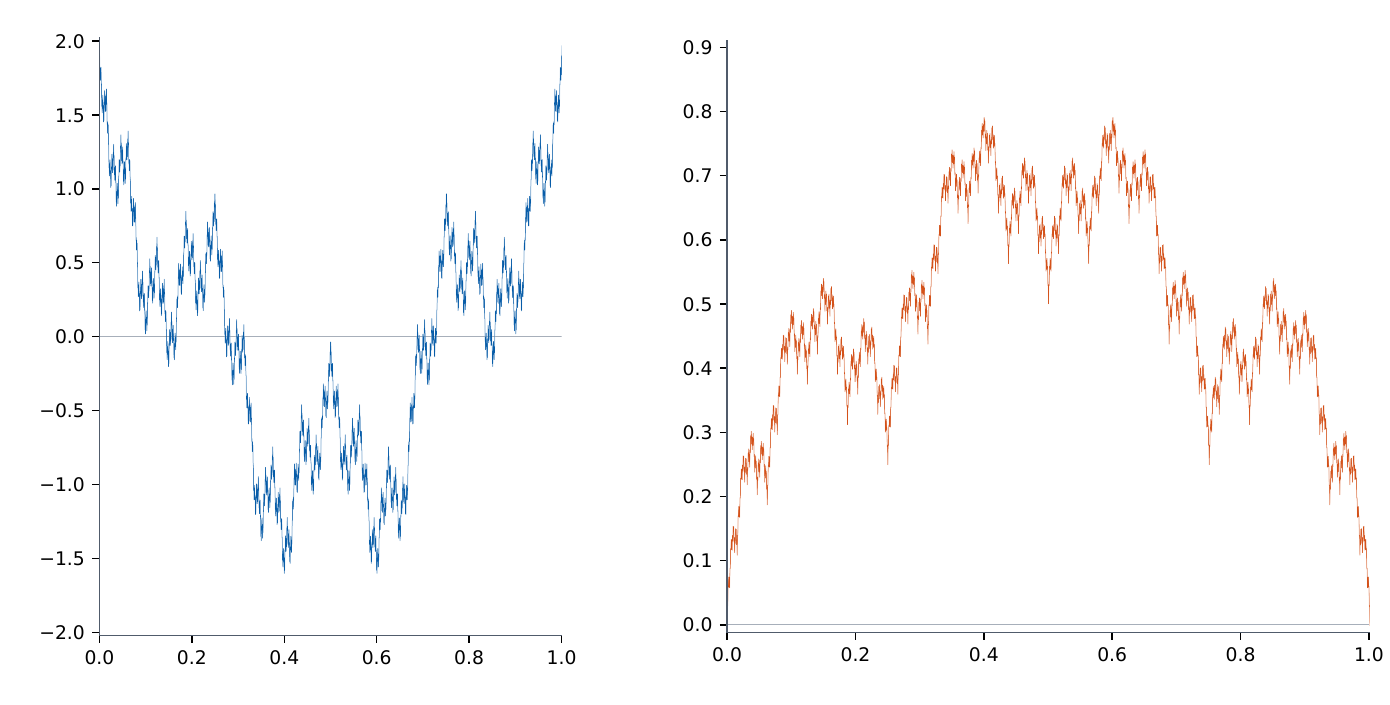}
    \caption{The classical Weierstrass and Takagi function graphs  with parameters $a=1/2$ and $b=4$.}
    \label{fig:weierstrass-takagi}
\end{figure}
    
    One particular object of interest relating to the aforementioned classes of functions $W_{a,b}$ and $T_{a,b}$ is the study of the corresponding function graphs
    $$
    G(W_{a,b}):=\{ (x, W_{a,b}(x)): \,x\in [0,1] \},
    $$ and similarly for $G(T_{a,b})$.
    Specifically, the graphs of these functions have received sustained attention within the fractal geometry, ergodic theory, and geometric measure theory communities for many years (see for instance \cite{FalcBook, FraserBook} for detailed expositions and references). A particularly challenging and longstanding program of research focuses on the determination of metric dimensions of the graphs of the Weierstrass functions. While the upper box dimension of $G(W_{a,b})$ has been known since the early 20th century (see \cite{HardyWeier}), its Hausdorff dimension was not conclusively determined until 2018 by Shen \cite{ShenWeier}, after a long list of partial results proved by many authors throughout many years (for instance see \cite{Hausd1934, Hausd1984, Hausd1989, Hausd2011, Hausd2014} for a non-exhaustive list). Determining exactly the Hausdorff dimension of  $G(W_{a,b})$ led to the celebrated dichotomy of Weierstrass functions based on their analytic properties and Hausdorff dimension, proved by Ren and Shen in their \textit{Inventiones Mathematicae} paper \cite{RenShenInventWeier}.

    A relatively new dimension notion, at least compared to the Hausdorff and the box-counting dimension, is that defined by Assouad in \cite{Assouad77} (see also \cite{Assouad79, Assouad83}). While the initial motivation for introducing the Assouad dimension was to study the embeddability of abstract metric spaces into Euclidean spaces, this dimension has been found to have surprising connections to topics in different areas. A non-exhaustive list of such instances includes the classification of spaces of homogeneous type (doubling) by the finiteness of the Assouad dimension \cite{hei:lectures}, a purely geometric criterion for admissibility of certain Hardy-type inequalities \cite{Lehr_Hardy_Aika_Assouad}, as well as one of the first recent breakthroughs towards resolving the Kakeya conjecture in $\R^3$ \cite{Assouad_Kakeya}, before it was announced by Wang and Zahl in \cite{Kakeya_Wang_Zahl_R3}. For more applications of the Assouad dimension within geometric measure theory, number theory, probability, and functional analysis we refer to the recent textbook of J.~M.~Fraser \cite{FraserBook}.

    Given the extensive interest on the various metric dimensions of Weierstrass function graphs, it is natural to ask what can be said about the Assouad dimension of $G(W_{a,b})$ and $G(T_{a,b})$. This is currently an open question posed by Fraser in \cite[Question 17.11.1]{FraserBook}, with partial progress in the case of $G(T_{a,b})$ for Takagi functions, and limited to no progress for the case of $G(W_{a,b})$ of Weierstrass functions. In particular, Yu  shows in \cite{Yu_Takagi} that for $b>3$ and certain parameters $a$, the Assouad dimension of $G(T_{a,b})$ exceeds the respective Hausdorff dimension, conjecturing that the former is equal to $2$ for all parameters $a\in (0,1)$ and $b\in (1/a,\infty)\cap \N$. Later, for $b=2$ and $a\in (1/2,1)$, Anttila, B\'ar\'any and K\"aenm\"aki \cite{TakagiDimA} give a precise formula for the Assouad dimension of $G(T_{a,2})$ that depends on the Hausdorff dimension of all slices of the graph, showing it is in fact strictly less than $2$, which disproves Yu's conjecture in this range of parameters. 
    
    Regarding the Assouad dimension of $G(W_{a,b})$, there are currently no non-trivial lower or upper bounds. However, Howroyd and Yu in \cite{Howroyd_Yu_Wiener_Process} show that the Assouad dimension of the graph of a Wiener process is almost surely equal to $2$, which naturally leads to the impression that Weierstrass functions should also have graphs of Assouad dimension equal to $2$, as natural deterministic representatives of these random functions. Moreover, Tyson and the author in \cite{ChronT_spec_holder} show that after countably many reflections of any given Weierstrass function graph, the resulting set is the graph of a H\"older function with the same exponent and of Assouad dimension equal to $2$. In addition, Feng and Fraser show in \cite{FengFraser_Holder_typical} that typical H\"older functions, in the Baire category sense, have graphs of Assouad dimension equal to $2$, further reinforcing the idea that the Assouad dimension of $G(W_{a,b})$ should also be $2$. Against all the aforementioned indications, the first main result of this paper establishes that the Assouad dimension of $G(W_{a,b})$ is strictly less than $2$ for all $a\in (0,1)$ and integers $b> a^{-1}$.

    \begin{Thm}
    \label{Thm:Aikawa-Weierstrass}
    Suppose $a\in (0,1)$ and $b\in\mathbb N$ is such that $ab>1$. Then
    $\dim_A G(W_{a,b})<2$.
    \end{Thm}

    In fact, the above follows from a more general theorem, which also settles Yu's conjecture for most of the remaining parameters. Given $a\in (0,1)$, an integer $b\in (1/a,\infty)$, and a Lipschitz $1$-periodic function $\phi:\R\to \R$, define the $1$-periodic \textit{$\phi-$generalized Weierstrass} function by
    $$
    W_{a,b}^\phi(x)
    =
    \sum_{j=0}^\infty a^j\phi(b^j x),
    \qquad 0\le x \le 1.
    $$
    We say that $\phi$ is the \textit{generating function} of $W_{a,b}^\phi$. Note that for $\phi =D$, where $D$ is the $1$-periodic sawtooth function defined earlier, we have $W_{a,b}^\phi=T_{a,b}$. 
    
    Before we state our second main result for generalized Weierstrass functions, we need to record a condition that is required for the generating functions $\phi$ we consider. Given $a,b,\phi$ as above, and $m\in \N$, set
    $$
    P_m(x)
    =
    \sum_{j=0}^{m-1}a^j\phi(b^jx), \qquad x\in \R
    $$
    and
    $$
    \kappa_\phi
    =
    \frac{\operatorname{Lip}(\phi)}{ab-1},
    $$ where $\operatorname{Lip} \phi$ denotes the Lipschitz constant of $\phi$ (see Section \ref{sec: background}). We say that a Lipschitz function $\phi:\R\to \R$ satisfies a  \textit{$b$-adic positive difference (PD) condition}  with data $(n_\phi, k_\phi, t_0)$  if there
    are integers $n_\phi\geq1$ and $1\leq k_\phi<b^{n_\phi}$, and a point
    $t_0\in[0,1-k_\phi b^{-n_\phi}]$ such that
    $$
    \left|
     P_{n_\phi}(t_0+k_\phi b^{-n_\phi})
    - P_{n_\phi}(t_0)
    \right|
    >
    \kappa_\phi k_\phi b^{-n_\phi}.
    $$

    A few comments are in order regarding the intuition behind this condition. A $b$-adic (PD) condition  asserts that some finite truncation $P_{n_\phi}$ has a $b$-adic secant of slope greater than $\kappa_\phi$. Since every rescaled lower-frequency contribution occurring in the iteration has Lipschitz constant at most $\kappa_\phi$, this oscillation cannot be canceled by the lower-frequencies. Moreover, the periodicity of $\phi$ implies that every term of index at least $n_\phi$ is invariant under the chosen $b$-adic translation, so the same separation persists in longer truncations. 
    
    We are now ready to state our second main result for generalized Weierstrass functions.

    \begin{Thm}\label{Thm:Aikawa-phi-Weierstrass}
        Suppose $a\in (0,1)$ and $b\in\mathbb N$ is such that $ab>1$. If $\phi:\R\to \R$ is a Lipschitz, $1$-periodic function satisfying a $b$-adic (PD) condition, then
    $\dim_A G(W_{a,b}^\phi)<2$.
    \end{Thm}

    In fact, we prove a quantitative version of Theorem  \ref{Thm:Aikawa-phi-Weierstrass}, which also implies a quantitative version of Theorem \ref{Thm:Aikawa-Weierstrass}(see Theorem \ref{Thm:Aikawa-Weierstrass QUANT}). This enables us to give an explicit upper bound depending solely on $a,b$ in the case of $W_{a,b}$ (see Theorem \ref{Thm:explicit-cosine-Assouad-bound}). As a result of Theorem  \ref{Thm:Aikawa-phi-Weierstrass}, we also show that Takagi functions   have Assouad dimension strictly less than $2$, answering in the negative Yu's conjecture from \cite{Yu_Takagi} in a wider range of parameters $a,b$ than in  \cite{TakagiDimA}.
    \begin{Thm}\label{Thm: Takagi Conj}
        Suppose $a\in (0,1)$ and $b\in \N$ is such that $ab>1$. Then $\dim_A G(T_{a,b})<2$.
    \end{Thm}

    For $b=2$ and $a\in (1/2,1)$, the methods employed by Anttila, B\'ar\'any and K\"aenm\"aki in \cite{TakagiDimA} rely on the dimension information of slices and the self-affine property of the graph $G(T_{a,2})$. While our approach is different, motivated by the importance of dimensions of slices exhibited in \cite{TakagiDimA}, as well as the general interest in the topic of dimensions of level sets (see for instance \cite{FraserBook} and the references therein), we also prove a quantitative bound on the Assouad dimension of all horizontal slices of $G(W_{a,b}^\phi)$.
    For
    $y\in\mathbb R$, we write
    $$
    \mathcal W_y
    =
    \{x\in[0,1]:W_{a,b}^\phi(x)=y\}
    $$ for the horizontal $y$-slice (also referred to as $y$-level set) of $W_{a,b}^\phi$.
    
    \begin{Thm}
    \label{Thm:uniform-level-set-estimate-QUAL}
    Suppose $a\in (0,1)$ and $b\in\mathbb N$ is such that $ab>1$. If $\phi:\R\to \R$ is a Lipschitz, $1$-periodic function with a $b$-adic (PD) condition, then $\dim_A \mathcal{W}_y<1$.
    \end{Thm}
    
    Our approach to prove Theorem \ref{Thm:Aikawa-phi-Weierstrass} relies on the equivalence of the Assouad dimension with the Aikawa dimension. The latter was first introduced by Aikawa in \cite{Aikawa_Riesz_capacity} for once again a different purpose. In particular, Aikawa's purpose for introducing this quantity was to study the Whitney quasiadditivity of the Riesz capacity of different sets $E\subset \R^n$. 
    However, it was later proved by Lehrb\"ack and Tuominen in \cite{LehrTuomAikawaAssouad_ORIG} that for subsets of  Euclidean spaces, this notion coincides with the Assouad dimension. This allows for counting arguments and cardinality bounds on relevant ``Aikawa intervals'', whose usefulness would not be as intuitive if we restricted our attention to the classical Assouad dimension definition.  We refer to Section \ref{sec: background} for rigorous definitions.

   The above theorems lead to several geometric and analytic implications for the corresponding Weierstrass and Takagi function graphs. Namely, a notion very closely related to the Assouad dimension is that of porosity, in the following sense. A set $E\subset \R^d$ is said to be \textit{porous} if there is $c\in (0,1)$ such that for every $x\in E$ and every $r\in (0,\diam E)$, the ball $B(x,r)$ contains a ball of radius $cr$ that does not intersect $E$. This notion was associated to the Assouad dimension by the result of Luukkainen \cite{Luukkainen_porous}. Namely, he proves that a non-empty set $E\subset \R^d$ is porous if, and only if, $\dim_A E<d$. This yields the following corollary as a direct consequence of Theorem \ref{Thm:Aikawa-Weierstrass}. 

 \begin{Cor}\label{cor: graph porous}
        For every $a\in (0,1)$ and integer $b\in (a^{-1},\infty)$, the graph $G(W_{a,b})$ is porous.
    \end{Cor}
    
    The above geometric property of the graph of the classical Weierstrass function reveals various  analytic characteristics of the function. In particular, it establishes connections of $W_{a,b}$ with certain Muckenhoupt weights and  Hardy-type inequalities that the graph $G(W_{a,b})$ admits. We refer to Section \ref{sec: Final Rem} for a more detailed discussion and precise statements.

    The paper is organized as follows. In Section~\ref{sec: background} we give the required rigorous definitions and fix the notation we follow in later sections. In Section~\ref{sec: counting intervals} we employ certain implications of the (PD) condition to count the relevant Aikawa intervals that we need for estimates of relevant integrals, and establish certain cardinality bounds. In Section~\ref{sec: dimA general W} we make use of the aforementioned cardinality bounds to prove the necessary upper bound on powers of distance integrals that lead to upper bound for the Aikawa dimension and, as a result, for the Assouad dimension of $G(W_{a,b}^\phi)$, proving a quantitative version of Theorem~\ref{Thm:Aikawa-phi-Weierstrass}. In Section~\ref{sec: dim level sets} we again use the cardinality estimates, in order to establish the relevant upper bounds on the Assouad dimension and, thus, the porosity of all horizontal  slices of $W_{a,b}^\phi$, proving Theorem~\ref{Thm:uniform-level-set-estimate-QUAL}. In Section~\ref{sec: dimA cosine W} we show that the cosine and sawtooth generating functions satisfy a sufficient condition for a $b$-adic (PD) condition, proving Theorems~\ref {Thm:Aikawa-Weierstrass} and \ref{Thm: Takagi Conj}. In the same section, we also give a constructive proof resulting in explicit upper bounds, and an easier to check sufficient condition for a generating function $\phi$ to satisfy a (PD) condition. In Section \ref{sec: Final Rem} we discuss further potential improvements of our results in certain cases, and various analytic implications of Corollary~\ref{cor: graph porous}.
    
\subsection*{Acknowledgments:} The author wishes to thank Roope Anttila, Sascha Troscheit, and Jeremy Tyson for the very interesting discussions on the topic. The author also thanks Qiyuan Gu for communicating the argument that extends the result for the classical Weierstrass and Takagi functions from the parameters of range $ab\geq \pi+1$ and $ab\geq 2$ to $ab>1$.

\section{Background and notation}\label{sec: background}

\subsection{Dimension notions}

We start this section by introducing useful notation. Given a collection $\mathcal V$ of subsets of $\R^d$, $d\in \N$, we denote by
$\#\mathcal V$ the cardinality of the collection $\mathcal V$. Given a non-empty set $J\subset \R$ we denote the $1$-Lebesgue measure of $J$ by $|J|$ for convenience. Given a non-empty set $E\subset\R^d$, $d\geq 2$, we denote by $\mathcal L^d$ the $d$-dimensional Lebesgue measure of $E$, and by $\diam E$ the diameter of $E$ under the Euclidean metric of $\R^d$, i.e.,
$$
\diam E= \sup\{|x-y|:\, x,y \in E\}.
$$ Given two non-empty sets $E,F\subset \R^d$, we  write
$$
\dist(E,F)=\inf \{ |x-y|: \, x\in E, y\in F \}.
$$ If $E=\{x\}$, then we write $\dist(x,F)=\dist(\{x\},F)$. Moreover, for $x\in \R^d$ and $r>0$, we denote by $B(x,r)$ the open Euclidean ball in $\R^d$ of center $x$ and radius $r$.

We now recall various relevant definitions. For a non-empty $E\subset \R^d$, the \textit{Assouad dimension} of $E$ is defined to be
$$
\dim_A E = \inf \left\{s>0 \,: \genfrac{}{}{0pt}{}{\exists\,C>0\ \text{s.t. } N_r(B(x,R)\cap E)\le C (R/r)^{s}}{\text{for all }0<r\le R\text{ and all }x \in E} \right\},
$$ where $N_r(B(x,R)\cap E)$ denotes the smallest number of sets of diameter at most $r$ needed to cover $B(x,R)\cap E$.

We recall the definition of the Aikawa dimension. Given a non-empty $E\subset \R^d$, the \textit{Aikawa dimension} of $E$, denoted by $\dim_{\mathrm{Ai}}E$, is the infimum of all $s>0$ for which there is $C_s>0$ such that
$$
\int_{B(z_0,R)} \dist(z, E)^{s-d}\,dz\leq C_s R^s,
$$ for all $z_0\in E$ and $R\in (0,\diam E)$. The following theorem of Lehrb\"ack-Tuominen \cite{LehrTuomAikawaAssouad_ORIG} is what motivated our current approach.
\begin{theoremA}[{\cite[Theorem 1.1]{LehrTuomAikawaAssouad_ORIG}}]
\label{Thm: Aikawa-Assouad-lehr}
Let $E\subset\mathbb R^d$ be a nonempty  set. Then $\dim_{\mathrm{Ai}}E
=
\dim_A E$. In particular,
for $\epsilon\in (0,1)$ and $d=2$, 
if there is $C_\epsilon>0$ such that
\begin{equation}
\label{eq:Aikawa-condition}
\int_{B(z_0,R)}
\operatorname{dist}(z,E)^{-\epsilon}\,dz
\leq
C_\epsilon R^{2-\epsilon},
\end{equation}
for every $z_0\in E$ and every $0<R<\operatorname{diam}E$, then
the Assouad dimension of $E$ is at most $2-\epsilon$.
\end{theoremA}

\subsection{Weierstrass functions}
In this subsection we fix the notation we follow regarding Weierstrass functions, truncated sums, certain uniform constants, and record a few elementary inequalities. Recall that a function $\phi:\R\to \R$ is \textit{Lipschitz} if there is $L>0$ such that 
$$
|\phi(x)-\phi(y)|\leq L |x-y|,
$$ for all $x,y\in \R$. The smallest such $L$ is called the \textit{Lipschitz constant} of $\phi$ and denoted by $\operatorname{Lip} \phi$.
Fix $a\in (0,1)$ and $b\in\mathbb N$ with $ab>1$,
for the rest of the manuscript. Let $\phi:\R\to\R$ be
a Lipschitz $1$-periodic function with
$$
B_\phi
=
\sup\{ |\phi(x)|:x\in [0,1] \},
\qquad
\kappa_\phi
=
\frac{\operatorname{Lip}(\phi)}{ab-1}.
$$and set
$$
W_{a,b}^\phi(x)
=
\sum_{j=0}^\infty a^j\phi(b^j x),
\qquad x\in [0,1].
$$
We say that $\phi$ is the \textit{generating function} of the Weierstrass function $W_{a,b}^\phi$. We denote the graph of the corresponding Weierstrass function by
$$
G_{a,b}^\phi=G(W_{a,b}^\phi)
=
\{(x,W_{a,b}^\phi(x)):0\leq x\leq1\}.
$$

If $a,b,\phi$ are fixed and clear in the given context, we often write $W=W_{a,b}^\phi$ and $G=G_{a,b}^\phi$.
Following this convention, for $m\in \N$, we also set
$$
P_m(x)=P_m(a,b,\phi;\,x)
=
\sum_{j=0}^{m-1}a^j\phi(b^jx),
$$
for the corresponding truncated sum, with $P_0=0$.

For an interval $I\subset [0,1]$ and a function $f:[0,1]\to\mathbb R$, we denote the oscillation of $f$ over $I$ by
$$
\operatorname{osc}_I f
=
\sup\{|f(x)-f(x')|:x,x'\in I\},
$$ and the Lipschitz constant of $f$ by 
$$
\operatorname{Lip}(f):=\sup \left\{\frac{|f(x)-f(x')|}{|x-x'|}:\, x,x'\in [0,1],\, x\neq x'\right\}.
$$
We record a few elementary estimates on $P_m$. For
every integer $m\geq0$ and every $x\in[0,1]$, we have
\begin{equation}
\label{eq:phi-tail-Pn}
\left|
W(x)- P_m(x)
\right|
\leq
\frac{B_\phi}{1-a}a^m.
\end{equation}
Also, due to $\phi$ being Lipschitz,
\begin{equation}
\label{eq:phi-Lipschitz-Pn}
\operatorname{Lip}(P_m)
\leq
\operatorname{Lip}(\phi)\sum_{j=0}^{m-1}(ab)^j
\leq
\kappa_\phi(ab)^m.
\end{equation}
Consequently, if $I$ is an interval of length at most $b^{-m}$, then
\begin{equation}
\label{eq:phi-oscillation-Pn}
\operatorname{osc}_I P_m
\leq
\kappa_\phi a^m.
\end{equation}

Denote the $b$-adic half-open intervals of level $n$  by
$$
\mathcal D_n
=
\left\{
\left[jb^{-n},(j+1)b^{-n}\right):
j\in\Z
\right\}.
$$ Given $A>0$ and $y\in \R$, define an interval $J\in \mathcal D_n$ to be an $(A,y)$-\textit{Aikawa interval} of level $n$ if
$$
\dist(y, P_n(\overline{J}))\leq A a^n.
$$These are the intervals that we aim to count and ``rule out'' in our estimate of the integral of $\dist(z,G)^{-\epsilon}$, by showing that there are not many of them, in a quantitative way (see Lemma \ref{Lem:phi-uniform-good-interval-estimate}).
Moreover, given a pair of integers
$0\leq m\leq n$, denote by
$$
\mathcal G^{\,n}_m(A,y)
=
\left\{
Q\in\mathcal D_m:
\text{there is } (A,y)-\text{Aikawa interval} \quad J\in\mathcal D_n\text{ with }J\subset Q
\right\}
$$ the collection of all $b$-adic intervals of level $m$ that contain Aikawa intervals of level $n$.
We may omit $A,y$ and simply write
${\mathcal G}^{\,n}_m$ if these values are fixed and clear from the context.

\section{Counting the Aikawa intervals}\label{sec: counting intervals}

Recall that a function $\phi:\R\to \R$ satisfies a  \textit{$b$-adic positive difference (PD) condition}  with data $(n_\phi, k_\phi, t_0)$  if there
are integers $n_\phi\geq1$ and $1\leq k_\phi<b^{n_\phi}$, and a point
$t_0\in[0,1-k_\phi b^{-n_\phi}]$ such that
\begin{equation}
\tag{PD}
\label{eq:phi-finite-block-transversality}
\left|
 P_{n_\phi}(t_0+k_\phi b^{-n_\phi})
- P_{n_\phi}(t_0)
\right|
>
\kappa_\phi k_\phi b^{-n_\phi}.
\end{equation}
%
%
Assume for the remainder of the section that $\phi$ is a Lipschitz $1$-periodic function generating $W=W_{a,b}^\phi$, and $\phi$ satisfies a $b$-adic (PD) condition with fixed data
$(n_\phi,k_\phi,t_0)$. Set
$$
\eta
=
\frac14
\left(
\left|
 P_{n_\phi}(t_0+k_\phi b^{-n_\phi})
- P_{n_\phi}(t_0)
\right|
-
\kappa_\phi k_\phi b^{-n_\phi}
\right)
>0.
$$
By continuity of the function
$$
F(t)=\left|
 P_{n_\phi}(t+k_\phi b^{-n_\phi})
- P_{n_\phi}(t)
\right|
$$ at $t_0$, where $F(t_0)= 4\eta +\kappa_\phi k_\phi b^{-n_\phi}$, there is a non-trivial closed interval
$$
I_\phi\subset(0,1-k_\phi b^{-n_\phi})
$$
with length less than $k_\phi b^{-n_\phi}$, such that
\begin{equation}
\label{eq:phi-uniform-finite-block-transversality}
\left|
 P_{n_\phi}(t+k_\phi b^{-n_\phi})
- P_{n_\phi}(t)
\right|
\geq
\kappa_\phi k_\phi b^{-n_\phi}+3\eta,
\end{equation}
for every $t\in I_\phi$. Notice that $I_\phi$ and
$I_\phi+k_\phi b^{-n_\phi}$ are disjoint due to $|I_\phi|<k_\phi b^{-n_\phi}$.

Choose and fix an integer $L\geq n_\phi$ sufficiently large that
\begin{equation}
\label{eq:phi-choice-of-L}
b^L|I_\phi|\geq2
\qquad\text{and}\qquad
a^L
\left(
\frac{B_\phi}{1-a}+\kappa_\phi
\right)
\leq
\frac{\eta}{4},
\end{equation}
and set
\begin{equation}
\label{eq:phi-def-sigma-theta}
\theta_\phi
=
\frac{-\log(1-|I_\phi|/2))}{L\log b}.
\end{equation}
Since $0<|I_\phi|<k_\phi b^{-n_\phi}<1$, we have
$0<|I_\phi|/2<1/2$, and hence $\theta_\phi\in(0,1)$, due to $b^L\geq 2$. In addition, note that
\begin{equation}
\label{eq:phi-sigma-theta-identity}
(1-|I_\phi|/2)b^L
=
b^{(1-\theta_\phi)L}.
\end{equation}
For $A>0$  fixed, choose an integer $M=M_A\geq L$ such that
\begin{equation}
\label{eq:phi-choice-of-M}
Aa^{M_A}
\leq
\frac{\eta}{4}.
\end{equation}
The integer $M_A$ depends only on $a,b,A, \phi$, and the (PD) condition data. If $A$ has been fixed and is clear from the context, we often write $M=M_A$.

We first count the intervals up to a uniform level that contain Aikawa sub-intervals.

\begin{Lem}
\label{Lem:phi-L-level-uniform-good-interval-estimate}
Fix
$A>0$. Then, for every $y\in\R$, for all integers $k,n$ with
$$
0\leq k\leq n-M_A,
$$
and for every interval $Q\in\mathcal D_k$, we have
\begin{equation}
\label{eq:phi-successive-level-count}
\#\left\{
Q'\in{\mathcal G}^{\,n}_{k+L}(A,y):Q'\subset Q
\right\}
\leq
(1-|I_\phi|/2)b^L
=
b^{(1-\theta_\phi)L}.
\end{equation}
\end{Lem}

\begin{proof}
Fix $A>0$ and let $M=M_A\geq L$ be as in
\eqref{eq:phi-choice-of-M}. Fix $y\in\R$, integers $k,n$ satisfying
$0\leq k\leq n-M$, and $Q\in\mathcal D_k$. For the rest of the proof
we write
$$
{\mathcal G}^{\,n}_{k+L}
=
{\mathcal G}^{\,n}_{k+L}(A,y).
$$
Write
$$
Q
=
\left[
\frac v{b^k},
\frac{v+1}{b^k}
\right),
\qquad
v\in\mathbb Z.
$$
The intervals in $\mathcal D_{k+L}$ contained in $Q$ are
$$
Q_\mu
=
\left[
\frac{v+\mu/b^L}{b^k},
\frac{v+(\mu+1)/b^L}{b^k}
\right),
\qquad
0\leq\mu<b^L.
$$
Denote the left endpoint of $Q_\mu$ by
$$
x_\mu
=
\frac{v+\mu/b^L}{b^k}.
$$

Suppose that
$Q_\mu\in{\mathcal G}^{\,n}_{k+L}$. There is then an $(A,y)$-Aikawa
interval $J\in\mathcal D_n$ such that
$$
J\subset Q_\mu
\qquad\text{and}\qquad
\operatorname{dist}
\bigl(y, P_n(\overline J)\bigr)
\leq
Aa^n.
$$
Since $ P_n(\overline J)$ is compact, there is
$x\in\overline J$ such that
$$
| P_n(x)-y|
\leq
Aa^n.
$$
Since $M\geq L$ and $k\leq n-M$, we have $k+L\leq n$. Using
\eqref{eq:phi-oscillation-Pn} and the above, we obtain
\begin{align}
a^{-k}
\left|
 P_{k+L}(x_\mu)-y
\right|
&\leq
a^{-k}| P_n(x)-y|
+
a^{-k}
\left|
 P_n(x)- P_{k+L}(x)
\right|
\nonumber\\
&\qquad
+
a^{-k}
\left|
 P_{k+L}(x)- P_{k+L}(x_\mu)
\right|
\nonumber\\
&\leq
Aa^{n-k}
+
\frac{B_\phi}{1-a}a^L
+
\kappa_\phi a^L
\nonumber\\
&\leq
Aa^M
+
a^L
\left(
\frac{B_\phi}{1-a}+\kappa_\phi
\right)
\nonumber\\
&\leq
\frac{\eta}{2},
\label{eq:phi-good-child-upper-bound}
\end{align}
where the last inequality follows from
\eqref{eq:phi-choice-of-L} and \eqref{eq:phi-choice-of-M}.

Define
$$
H(t)
=
a^{-k}
 P_{k+L}
\left(
\frac{v+t}{b^k}
\right),
\qquad
0\leq t\leq1.
$$
We then have
$$
H(t)
=
g(t)+ P_L(t),
$$
where
$$
g(t)
=
a^{-k}
 P_k
\left(
\frac{v+t}{b^k}
\right).
$$
Indeed, for every $0\leq j<L$, the $1$-periodicity of $\phi$ and the
fact that $b^jv$ is an integer give
$$
\phi\left(
b^{k+j}\frac{v+t}{b^k}
\right)
=
\phi(b^jv+b^jt)
=
\phi(b^jt),
$$
which implies
$$
P_L(t)=\sum_{j=0}^{L-1} a^{j}\phi(b^j t)=\sum_{j=0}^{L-1} a^{j}\phi\left(
b^{k+j}\frac{v+t}{b^k}\right)=
\sum_{j=k}^{k+L-1} a^{j-k}\phi\left(
b^{j}\frac{v+t}{b^k}\right)=H(t)-g(t)
$$
Since $x_\mu=(v+\mu/b^L)/b^{k}$, it follows from \eqref{eq:phi-good-child-upper-bound} and the definition of $H$ that
\begin{equation}
\label{eq:phi-H-value-condition}
Q_\mu\in{\mathcal G}^{\,n}_{k+L}
\quad\text{implies}\quad
\left|
H\left(\frac{\mu}{b^L}\right)-a^{-k}y
\right|
\leq
\frac{\eta}{2}.
\end{equation}
After an application of triangle inequality, the above also yields that
\begin{equation}
\label{eq:phi-H-two-values}
Q_{\mu_1},Q_{\mu_2}\in{\mathcal G}^{\,n}_{k+L}\quad\text{implies}\quad\left|
H\left(\frac{\mu_1}{b^L}\right)
-
H\left(\frac{\mu_2}{b^L}\right)
\right|
\leq
\eta.
\end{equation}

By \eqref{eq:phi-Lipschitz-Pn}, we have
\begin{equation}
\label{eq:phi-Lipschitz-g}
\operatorname{Lip}(g)
\leq
a^{-k}b^{-k}\operatorname{Lip}( P_k)
\leq
\kappa_\phi.
\end{equation}
On the other hand, for every $j\geq n_\phi$ we have
$k_\phi b^{j-n_\phi}\in\mathbb Z$.
Therefore, for every $L\geq n_\phi$ and
$t\in[0,1-k_\phi b^{-n_\phi}]$, the
$1$-periodicity of $\phi$ gives
\begin{equation}
\label{eq:phi-block-translation}
 P_L(t+k_\phi b^{-n_\phi})- P_L(t)
=
 P_{n_\phi}(t+k_\phi b^{-n_\phi})
- P_{n_\phi}(t).
\end{equation}
Combining \eqref{eq:phi-uniform-finite-block-transversality},
\eqref{eq:phi-Lipschitz-g}, and \eqref{eq:phi-block-translation}, we
obtain
\begin{align}
|H(t+k_\phi b^{-n_\phi})-H(t)|
&\geq
\left|
 P_L(t+k_\phi b^{-n_\phi})- P_L(t)
\right|
-
|g(t+k_\phi b^{-n_\phi})-g(t)|
\nonumber\\
&\geq
\kappa_\phi k_\phi b^{-n_\phi}+3\eta
-\kappa_\phi k_\phi b^{-n_\phi}
\nonumber\\
&=
3\eta
\label{eq:phi-H-separated-values}
\end{align}
for every $t\in I_\phi$.

Set $q
=
k_\phi b^{L-n_\phi}$.
Then, since $k_\phi<b^{n_\phi}$, $q$ is an integer satisfying $1\leq q<b^L$ and
\begin{equation}\label{eq: q def and k_phi relation}
    \frac q{b^L}
=
k_\phi b^{-n_\phi}.
\end{equation}
Consider the set of indices that correspond to endpoints of intervals in $\mathcal D_L$ contained in $I_\phi$, i.e.,
$$
\mathcal V
=
\left\{
\mu\in\mathbb Z:
0\leq\mu<b^L
\text{ and }
\frac{\mu}{b^L}\in I_\phi
\right\}.
$$
Note that by \eqref{eq:phi-choice-of-L} and $ab>1$, we have $|I_\phi|\geq 2b^{-L}$, and so $\mathcal V\neq \emptyset$. In addition, since $I_\phi\subset(0,1-k_\phi b^{-n_\phi})$, both $\mu$ and
$\mu+q$ lie between
$0$ and $b^L-1$ for all $\mu\in\mathcal V$. Moreover, by definition of $\mathcal V$, we have
$$
(\#\mathcal V+1)b^{-L}\geq |I_\phi|.
$$ This and $|I_\phi|\geq 2b^{-L}$ imply that
\begin{equation}
\label{eq:phi-number-of-indices}
\#\mathcal V
\geq
b^L|I_\phi|-1
\geq
\frac{|I_\phi|}{2}b^L.
\end{equation}

The index sets
$$
\{\mu,\mu+q\},
\qquad
\mu\in\mathcal V,
$$
are mutually disjoint. Indeed, their first indices correspond to  $b$-adic intervals of level $L$ with left endpoints in $I_\phi$, while their second indices correspond to  $b$-adic intervals of level $L$ with left endpoints in
$I_\phi+k_\phi b^{-n_\phi}$, and
$I_\phi\cap(I_\phi+k_\phi b^{-n_\phi})=\varnothing$.

For every $\mu\in\mathcal V$, applying
\eqref{eq:phi-H-separated-values} at $t=\mu/b^L$ and using \eqref{eq: q def and k_phi relation} gives
$$
\left|
H\left(\frac{\mu}{b^L}\right)
-
H\left(\frac{\mu+q}{b^L}\right)
\right|
\geq
3\eta.
$$
If both $Q_\mu$ and $Q_{\mu+q}$ belonged to
${\mathcal G}^{\,n}_{k+L}$, then
\eqref{eq:phi-H-two-values} would imply that the left-hand side is at
most $\eta$, which is a contradiction. Thus, for each
$\mu\in\mathcal V$, at least one of $Q_\mu$ and $Q_{\mu+q}$ does not
belong to ${\mathcal G}^{\,n}_{k+L}$. Since the corresponding
pairs of indices are mutually disjoint, \eqref{eq:phi-number-of-indices}
shows that at least $(|I_\phi|/2)b^L$ of the intervals $Q_\mu$ do not
belong to ${\mathcal G}^{\,n}_{k+L}$. Consequently,
\begin{align*}
\#\left\{
Q'\in{\mathcal G}^{\,n}_{k+L}:Q'\subset Q
\right\}
&\leq
b^L-(|I_\phi|/2)b^L\\
&=
(1-|I_\phi|/2)b^L\\
&=
b^{(1-\theta_\phi)L},
\end{align*}
where the final equality follows from
\eqref{eq:phi-sigma-theta-identity}. This completes the proof.
\end{proof}

Given arbitrary $A>0$, $y\in \R$, we record a hereditary property of the collections $\mathcal G_\ell^n(A,y)$ that we use in following arguments. Namely, if
$Q'\in{\mathcal G}^{\,n}_\ell(A,y)$ and $Q\in\mathcal D_m$ is
an ancestor of $Q'$, where $0\leq m\leq\ell$, then
$Q\in{\mathcal G}^{\,n}_m(A,y)$. Indeed, if
$J\in\mathcal D_n$ is an $(A,y)$-Aikawa sub-interval of level $n$ of
$Q'\in{\mathcal G}^{\,n}_\ell(A,y)$, then due to
$$
J\subset Q'\subset Q,
$$
it is also an $(A,y)$-Aikawa sub-interval of $Q$, implying
$Q\in{\mathcal G}^{\,n}_m(A,y)$. 

We next iterate the previous lemma accordingly, in order to obtain similar uniform cardinality bounds for other levels.

\begin{Lem}
\label{Lem:phi-uniform-good-interval-estimate}
Suppose that $A>0$. Then there is $C\geq1$, depending only on
$a,b,\phi,A$, such that for every triplet of integers
$0\leq p\leq N\leq n$, for every interval $K\in\mathcal D_p$, and
for every $y\in\R$, we have
\begin{equation}
\label{eq:phi-uniform-good-interval-count}
\#\left\{
I\in{\mathcal G}^{\,n}_N(A,y):I\subset K
\right\}
\leq
C b^{(1-\theta_\phi)(N-p)},
\end{equation}
where $\theta_\phi\in(0,1)$ is as in
\eqref{eq:phi-def-sigma-theta}.
\end{Lem}

\begin{proof}
Fix $A>0$, and let $M=M_A\geq L$ be as in
\eqref{eq:phi-choice-of-M}. Fix also $K\in\mathcal D_p$,
 $y\in\R$, and
$$
0\leq p\leq N\leq n.
$$
For every integer $0\leq m\leq n$, write
$$
{\mathcal G}^{\,n}_m
=
{\mathcal G}^{\,n}_m(A,y)
$$
for the rest of the proof.

We shall use Lemma \ref{Lem:phi-L-level-uniform-good-interval-estimate} iteratively to show \eqref{eq:phi-uniform-good-interval-count}. 
Let $s$ be the largest nonnegative integer such that, for every integer
$j$ with $0\leq j<s$,
\begin{equation}
\label{eq:phi-successive-level-conditions}
p+jL\leq n-M
\qquad\text{and}\qquad
p+(j+1)L\leq N.
\end{equation}
The value $s=0$ is allowed. In particular, $s=0$ if at least one of the above inequalities fails for $j=s=0$. 

For every $0\leq j\leq s$, define
$$
\mathcal F_j
=
\left\{
Q\in{\mathcal G}^{\,n}_{p+jL}:Q\subset K
\right\}.
$$
Since $K$ is the only interval in $\mathcal D_p$ contained in itself,
we trivially have $\#\mathcal F_0\leq1$.

For every $0\leq j<s$, the level $(p+jL)$ ancestor of every interval
in $\mathcal F_{j+1}$ belongs to $\mathcal F_j$ by the hereditary
property of the collections $\mathcal{G}^n_\ell(A,y)$. Applying \eqref{eq:phi-successive-level-count} to each
interval in $\mathcal F_j$, which is possible due to $j<s$ and
\eqref{eq:phi-successive-level-conditions}, we obtain
$$
\#\mathcal F_{j+1}
\leq
(1-|I_\phi|/2)b^L\#\mathcal F_j.
$$
Applying the above inductively yields
\begin{equation}
\label{eq:phi-count-after-successive-levels}
\#\mathcal F_s
\leq
\left((1-|I_\phi|/2)b^L\right)^s.
\end{equation}

If $s\geq1$, the second inequality in
\eqref{eq:phi-successive-level-conditions}, with $j=s-1$, gives
$p+sL\leq N$; if $s=0$, this follows from $p\leq N$. By the maximality
of $s$, at least one of the inequalities
$$
p+sL\leq n-M,
\qquad
p+(s+1)L\leq N
$$
must fail. Therefore, either
$$
p+sL>n-M,
$$
or
$$
p+(s+1)L>N.
$$
In the former case, since $N\leq n$, we have
$$
N-p-sL
\leq
n-p-sL
<
M.
$$
In the latter case,
$$
N-p-sL
<
L
\leq
M.
$$
Thus, for any such maximal integer $s\geq 0$ we have
\begin{equation}
\label{eq:phi-remaining-levels}
0\leq N-p-sL<M.
\end{equation}

If $I\in{\mathcal G}^{\,n}_N$ is contained in $K$, then
its level $(p+sL)$ ancestor exists due to the second inequality in \eqref{eq:phi-successive-level-conditions}  and it belongs to $\mathcal F_s$ by the
aforementioned hereditary property. Each interval in $\mathcal F_s$ contains 
$$
\frac{b^{-(p+sL)}}{b^{-N}}=b^{N-p-sL}
$$
intervals from $\mathcal D_N$. It follows from
\eqref{eq:phi-count-after-successive-levels} that
\begin{align*}
\#\left\{
I\in{\mathcal G}^{\,n}_N:I\subset K
\right\}
&\leq
\left((1-|I_\phi|/2)b^L\right)^s
b^{N-p-sL}\\
&=
b^{N-p}(1-|I_\phi|/2)^s.
\end{align*}
Note that by \eqref{eq:phi-remaining-levels} we have
$$
s
\geq
\frac{N-p-M}{L}.
$$
Since $0<1-|I_\phi|/2<1$, the above inequalities yield
\begin{align*}
\#\left\{
I\in{\mathcal G}^{\,n}_N:I\subset K
\right\}
&\leq
b^{N-p}
(1-|I_\phi|/2)^{(N-p-M)/L}\\
&=
(1-|I_\phi|/2)^{-M/L}
b^{N-p}
(1-|I_\phi|/2)^{(N-p)/L}\\
&=
(1-|I_\phi|/2)^{-M/L}
b^{(1-\theta_\phi)(N-p)},
\end{align*}
where the last equality follows from
\eqref{eq:phi-def-sigma-theta}. Therefore
\eqref{eq:phi-uniform-good-interval-count} holds with
$$
C
=
(1-|I_\phi|/2)^{-M/L}.
$$
Since $M$ depends only on $a,b,\phi,A$, the same is true for $C$.
\end{proof}

\section{Assouad dimension and general Weierstrass functions}\label{sec: dimA general W}
In this section we prove a quantitative version of Theorem \ref{Thm:Aikawa-phi-Weierstrass}, using the Aikawa characterization of the Assouad dimension. For $a,b,\phi$ and $W=W_{a,b}^\phi$ as in the previous section, write $G=G(W)$ and
$$
G_t
=
\{z\in\mathbb R^2:\operatorname{dist}(z,G)<t\},
$$ for all $t>0$. Also, fix $\theta=\theta_\phi\in (0,1)$ as in \eqref{eq:phi-def-sigma-theta} for the rest of the section.
We aim to decompose the integral of the distance function over the arbitrary $B(z_0, R)$ into a sum of integrals over annular regions $G_{b^{-j}}\setminus G_{b^{-j-1}}$ around $G$ within the ball. We then use the following estimates that are derived from the upper bound on the number of Aikawa intervals.

\begin{Lem}
\label{Lem:phi-uniform-Aikawa-shell-estimate}
There is $C\geq 1$ depending only on $a, b, \phi$ such that for every $\epsilon\in (0,\theta_\phi)$, for every $z_0\in G$, and for every $R\in (0,1)$ we have 
\begin{equation}
\label{eq:phi-uniform-Aikawa-shell-estimate}
\int_{B(z_0,R)\cap
\left(
 G_{b^{-j}}\setminus
 G_{b^{-j-1}}
\right)}
\operatorname{dist}(z, G)^{-\epsilon}\,dz
\leq
Cb^{1+\epsilon}R
b^{-(1-\theta_\phi)p}
b^{-j(\theta_\phi-\epsilon)},
\end{equation} for every integer $j\geq p=\lceil -\log_b R\rceil$.
\end{Lem}

\begin{proof}
Fix $\epsilon\in(0,\theta_\phi)$, $z_0=(x_0, W(x_0))\in G$,
and $R\in(0,1)$. Let
$$
p=\left\lceil-\log_bR\right\rceil
$$
and fix an integer $j\geq p$. Then
\begin{equation}
\label{eq:phi-choice-of-p}
b^{-p}\leq R<b^{-p+1}.
\end{equation}
Choose the unique integer $m\geq1$ such that
\begin{equation}
\label{eq:phi-Aikawa-shell-scales}
a^m
\leq
b^{-j}
<
a^{m-1}.
\end{equation}
If $j>m$, then $ab>1$ gives
$$
b^{-j}
<
b^{-m}
<
a^m
\leq
b^{-j},
$$
which is impossible. Thus, $j\leq m$.

Set
$$
A_\phi
=
\frac1a+
\frac{B_\phi}{1-a}.
$$
For $y\in\R$, set
$$
 E_y^j
=
\left\{
x\in\R:
(x,y)\in
B(z_0,R)\cap
\left(
 G_{b^{-j}}\setminus
 G_{b^{-j-1}}
\right)
\right\}.
$$to be the projection of the horizontal $y$-slice of $B(z_0,R)\cap
(G_{b^{-j}}\setminus G_{b^{-j-1}})$ into the $x$-axis.
If $x\in E_y^j$, then there is $(u_x, W(u_x))\in G$ with
$$
|(x,y)-(u_x, W(u_x))|<b^{-j},
$$
which implies
\begin{equation}
\label{eq:phi-choice-of-ux}
|x-u_x|<b^{-j}
\qquad\text{and}\qquad
|y- W(u_x)|<b^{-j}.
\end{equation}
Let $J\in\mathcal D_m$ be the unique level $m$ interval containing $u_x$. By
\eqref{eq:phi-tail-Pn}, \eqref{eq:phi-Aikawa-shell-scales}, and
\eqref{eq:phi-choice-of-ux}, we have
\begin{align}
\operatorname{dist}
\bigl(y, P_m(\overline J)\bigr)
&\leq
|y- W(u_x)|
+
| W(u_x)- P_m(u_x)|
\nonumber\\
&<
b^{-j}
+
\frac{B_\phi}{1-a}a^m
\nonumber\\
&<
\left(
\frac1a+\frac{B_\phi}{1-a}
\right)a^m
\nonumber\\
&=
A_\phi a^m.
\label{eq:phi-tube-good-interval}
\end{align}

By \eqref{eq:phi-choice-of-p}, $j\geq p$, and
\eqref{eq:phi-choice-of-ux}, we have
$$
|u_x-x_0|
\leq
|u_x-x|+|x-x_0|
<
b^{-j}+R
\leq
2R.
$$
The interval $[x_0-2R,x_0+2R]$ meets at most $10b$ intervals in
$\mathcal D_p$. Let $I\in\mathcal D_j$ be the level $j$ ancestor of
$J$. The level $p$ ancestor of $I$ is one of these at most $10b$
intervals. Moreover, \eqref{eq:phi-tube-good-interval} shows that
$$
I\in{\mathcal G}^{\,m}_j(A_\phi,y).
$$
Since $p\leq j\leq m$, applying
Lemma~\ref{Lem:phi-uniform-good-interval-estimate},  with
$N=j$ and $n=m$ on each of these level $p$ intervals shows that
$u_x$ lies in the union of at most
$$
10bC_0b^{(1-\theta_\phi)(j-p)}
$$
intervals from $\mathcal D_j$. Denote this union by
${\mathcal U}=\cup_i U_i$.

Therefore, for every $x\in E_y^j$, there is 
$u_x\in{\mathcal U}$ such that $|x-u_x|<b^{-j}$. Thus,
$ E_y^j$ is contained in the $b^{-j}$-neighborhood of
${\mathcal U}$. Also, each $U_i$ has length $b^{-j}$, and its
$b^{-j}$-neighborhood has length at most $3b^{-j}$. Thus, we have
\begin{equation}
\label{eq:phi-Aikawa-shell-horizontal-estimate}
| E_y^j|
\leq
30bC_0b^{-j}
b^{(1-\theta_\phi)(j-p)}.
\end{equation}

On
$ G_{b^{-j}}\setminus G_{b^{-j-1}}$ we have
$$
\operatorname{dist}(z, G)
\geq
b^{-j-1},
$$
and hence
$$
\operatorname{dist}(z, G)^{-\epsilon}
\leq
b^{(j+1)\epsilon}.
$$
Moreover, if $ E_y^j$ is nonempty, then $|y- W(x_0)|<R$. Therefore,
by 
\eqref{eq:phi-Aikawa-shell-horizontal-estimate} we have
\begin{align*}
&
\int_{B(z_0,R)\cap
\left(
 G_{b^{-j}}\setminus
 G_{b^{-j-1}}
\right)}
\operatorname{dist}(z, G)^{-\epsilon}\,dz
\\
&\qquad\leq
b^{(j+1)\epsilon}
\int_{ W(x_0)-R}^{ W(x_0)+R}
| E_y^j|\,dy
\\
&\qquad\leq
60b^{1+\epsilon}C_0R
b^{-(1-\theta_\phi)p}
b^{-j(\theta_\phi-\epsilon)},
\end{align*}
which completes the proof.
\end{proof}

We now prove a quantitative version of Theorem \ref{Thm:Aikawa-phi-Weierstrass}.

\begin{Thm}
\label{Thm:Aikawa-Weierstrass QUANT}
Suppose $a\in (0,1)$ and $b\geq2$ is an integer with $ab>1$. Let
$\phi:\R\to\R$ be a Lipschitz $1$-periodic function satisfying a $b$-adic 
\eqref{eq:phi-finite-block-transversality} condition. Then
$$
\dim_A G(W_{a,b}^\phi)
\leq
2-\theta_\phi<2,
$$
where $\theta_\phi\in(0,1)$ is as in
\eqref{eq:phi-def-sigma-theta}.
\end{Thm}

\begin{proof}
Let $A_\phi$ and $C_0$ be as in
Lemma~\ref{Lem:phi-uniform-Aikawa-shell-estimate}.
Due to Theorem~\ref{Thm: Aikawa-Assouad-lehr}, it is enough to show
that for every $0<\epsilon<\theta_\phi$, there is
$C_\epsilon\geq1$, depending only on $a,b,\phi,\epsilon$, such that
\begin{equation}
\label{eq:phi-final-Aikawa-estimate}
\int_{B(z_0,R)}
\operatorname{dist}(z, G)^{-\epsilon}\,dz
\leq
C_\epsilon R^{2-\epsilon},
\end{equation}
for every $z_0\in G$ and every
$0<R<\operatorname{diam} G$.

Fix $z_0\in G$ and suppose first that $0<R<1$. Let
$$
p
=
\left\lceil-\log_bR\right\rceil,
$$
so that
$$
b^{-p}
\leq
R
<
b^{-p+1}.
$$
Since $z_0\in G$, the above yields
$$
\operatorname{dist}(z, G)
\leq
|z-z_0|
<
R
<
b^{-p+1},
$$
for every $z\in B(z_0,R)$. For every integer $N>p$, the sets
$$
 G_{b^{-j}}\setminus
 G_{b^{-j-1}},
\qquad
p-1\leq j<N,
$$
are pairwise disjoint, and
$$
\bigcup_{j=p-1}^{N-1}
\left(
 G_{b^{-j}}\setminus
 G_{b^{-j-1}}
\right)
=
 G_{b^{-p+1}}\setminus
 G_{b^{-N}}.
$$
By an application of the monotone convergence theorem, and the fact that
$\mathcal L^2( G)=0$, we ahve
\begin{align}
&
\int_{B(z_0,R)}
\operatorname{dist}(z, G)^{-\epsilon}\,dz
\nonumber\\
&\qquad=
\int_{B(z_0,R)\cap
\left(
 G_{b^{-p+1}}\setminus
 G_{b^{-p}}
\right)}
\operatorname{dist}(z, G)^{-\epsilon}\,dz
\nonumber\\
&\qquad\quad+
\lim_{N\to\infty}
\sum_{j=p}^{N-1}
\int_{B(z_0,R)\cap
\left(
 G_{b^{-j}}\setminus
 G_{b^{-j-1}}
\right)}
\operatorname{dist}(z, G)^{-\epsilon}\,dz.
\label{eq:phi-Aikawa-shell-decomposition}
\end{align}

On
$ G_{b^{-p+1}}\setminus G_{b^{-p}}$ we have
$$
\operatorname{dist}(z, G)^{-\epsilon}
\leq
b^{p\epsilon}.
$$
Since $R<b^{-p+1}$, it follows that
\begin{align*}
&
\int_{B(z_0,R)\cap
\left(
 G_{b^{-p+1}}\setminus
 G_{b^{-p}}
\right)}
\operatorname{dist}(z, G)^{-\epsilon}\,dz
\\
&\qquad\leq
\pi R^2b^{p\epsilon}
\leq
\pi b^\epsilon R^{2-\epsilon}.
\end{align*}

For every integer $N>p$,
Lemma~\ref{Lem:phi-uniform-Aikawa-shell-estimate} gives
\begin{align*}
&
\sum_{j=p}^{N-1}
\int_{B(z_0,R)\cap
\left(
 G_{b^{-j}}\setminus
 G_{b^{-j-1}}
\right)}
\operatorname{dist}(z, G)^{-\epsilon}\,dz
\\
&\qquad\leq
60b^{1+\epsilon}C_0R
b^{-(1-\theta_\phi)p}
\sum_{j=p}^{N-1}b^{-j(\theta_\phi-\epsilon)}
\\
&\qquad\leq
\frac{60b^{1+\epsilon}C_0}
{1-b^{-(\theta_\phi-\epsilon)}}
R b^{-p(1-\epsilon)}
\\
&\qquad\leq
\frac{60b^{1+\epsilon}C_0}
{1-b^{-(\theta_\phi-\epsilon)}}
R^{2-\epsilon},
\end{align*}
where the last inequality follows from $b^{-p}\leq R$. Letting
$N\to\infty$ in this estimate and using
\eqref{eq:phi-Aikawa-shell-decomposition} proves
\eqref{eq:phi-final-Aikawa-estimate} for $R<1$, with a constant
$C_\epsilon'\geq1$ depending only on $a,b,\phi,\epsilon$.

Suppose now that
$$
1\leq R<\operatorname{diam} G.
$$
Since $ G$ is compact, there are
$$
z_1,\ldots,z_{M_W}\in G
$$
such that
$$
 G
\subset
\bigcup_{i=1}^{M_W}B(z_i,b^{-2}).
$$
The integer $M_W$ depends only on the fixed parameters
$a,b,\phi$. It follows that
$$
 G_{b^{-2}}
\subset
\bigcup_{i=1}^{M_W}B(z_i,2b^{-2}).
$$
Since $b\geq2$, we have $2b^{-2}<1$. We may therefore use
\eqref{eq:phi-final-Aikawa-estimate}, which we have proved
for all radii less than $1$. Thus,
\begin{align*}
\int_{B(z_0,R)\cap G_{b^{-2}}}
\operatorname{dist}(z, G)^{-\epsilon}\,dz
&\leq
\sum_{i=1}^{M_W}
\int_{B(z_i,2b^{-2})}
\operatorname{dist}(z, G)^{-\epsilon}\,dz
\\
&\leq
M_W C_\epsilon'(2b^{-2})^{2-\epsilon}
\\
&\leq
M_W C_\epsilon'(2b^{-2})^{2-\epsilon}
R^{2-\epsilon}.
\end{align*}
On the other hand, if
$z\in B(z_0,R)\setminus G_{b^{-2}}$, then
$\operatorname{dist}(z, G)\geq b^{-2}$. Therefore,
due to $R<\operatorname{diam} G$ we have
\begin{align*}
\int_{B(z_0,R)\setminus G_{b^{-2}}}
\operatorname{dist}(z, G)^{-\epsilon}\,dz
&\leq
\pi b^{2\epsilon}R^2
\\
&\leq
\pi b^{2\epsilon}
(\operatorname{diam} G)^\epsilon
R^{2-\epsilon}.
\end{align*}
Adding the above two inequalities proves \eqref{eq:phi-final-Aikawa-estimate}  for $R\geq1$.

Therefore,
$$
\dim_A G
\leq
2-\epsilon
$$
for every $0<\epsilon<\theta_\phi$. Letting
$\epsilon\to\theta_\phi$ completes the proof.
\end{proof}

\begin{Rem}
    One could follow the proof of Theorem A from \cite{LehrTuomAikawaAssouad_ORIG} and appropriately adjust the arguments above to obtain suitable covers for the graph $G(W)$, which could be used to prove Theorem \ref{Thm:Aikawa-Weierstrass QUANT} using the Assouad dimension definition involving the covering number $N_r(B(x,R)\cap G)$. However, this approach would not result in a sharper upper bound. Moreover, the main counting argument that led to this proof was motivated by the notion of Aikawa dimension and we therefore find it  more appropriate to include the proof above instead of reproducing the proof of Theorem A to employ covering arguments.
\end{Rem}

\section{Assouad dimension and porosity of horizontal slices}
\label{sec: dim level sets}

In this section we prove a quantitative version of Theorem \ref{Thm:uniform-level-set-estimate-QUAL}, which also establishes the porosity of all horizontal slices. Let $a,b,\phi$, $W=W_{a,b}^\phi$, $G=G(W)$ be as in the previous section, and let $\theta=\theta_\phi\in (0,1)$ be as in \eqref{eq:phi-def-sigma-theta}. The main tool is again the bound on the number of Aikawa intervals in Lemma~\ref{Lem:phi-uniform-good-interval-estimate}, but in this case we employ the usual definition of Assouad dimension.  For
$y\in\mathbb R$, we write
$$
\mathcal W_y
=
\{x\in[0,1]:W(x)=y\}
$$ for the horizontal $y$-slice of $W$.

\begin{Thm}
\label{Thm:uniform-level-set-estimate}
Suppose $a\in (0,1)$ and $b\in\mathbb N$ is such that $ab>1$, and $\phi:\R\to \R$ is a Lipschitz, $1$-periodic function with a $b$-adic (PD) condition. For every
$y\in\mathbb R$, we have
$$
\dim_A \mathcal W_y\leq1-\theta,
$$
where $\theta=\theta_\phi\in(0,1)$ is as in \eqref{eq:phi-def-sigma-theta}. In particular, $\mathcal W_y$ is a porous subset of $\R$.
\end{Thm}

\begin{proof}
Fix $y\in\mathbb R$. If $\mathcal W_y=\emptyset$ then the inequality is trivially true, so suppose $\mathcal W_y\neq \emptyset$. Set
$$
A=\frac{B_\phi}{1-a},
$$
and let $C_0\geq1$ be the constant in
Lemma~\ref{Lem:phi-uniform-good-interval-estimate} corresponding to the chosen
$A$.  This constant is independent of $y$.  We aim to show that 
\begin{equation}
\label{eq:uniform-level-set-estimate}
N_r\bigl(B(x,R)\cap \mathcal W_y\bigr)
\leq
C\left(\frac Rr\right)^{1-\theta},
\end{equation} for any $x\in \mathcal W_y$, and any $0<r\leq R$. Note that $B(x,R)$ is the interval $(x-R, x+R)$, since $\mathcal W_y\subset \R$.

Let $x\in \mathcal W_y$, and first suppose that
$0<r\leq R<1$.  Let $p,n\in\mathbb N$ be the unique integers such
that
$$
b^{-p}\leq R<b^{-p+1}
\qquad\text{and}\qquad
b^{-n}\leq r<b^{-n+1}.
$$
Since $r\leq R$, we have $p\leq n$.  The interval
$B(x,R)\subset\mathbb R$ meets at most $10b$ intervals from
$\mathcal D_p$.  Indeed, its length is $2R<2b^{-p+1}$, and every
interval in $\mathcal D_p$ has length $b^{-p}$.

Let $J\in\mathcal D_n$ satisfy
$$
J\cap B(x,R)\cap \mathcal W_y\neq\varnothing,
$$
and choose $u$ in this intersection.  By \eqref{eq:phi-tail-Pn},
$$
\operatorname{dist}\bigl(y,P_n(\overline J)\bigr)
\leq
|y-P_n(u)|
=
|W(u)-P_n(u)|
\leq
\frac{B_\phi a^n}{1-a}
=
Aa^n.
$$
Consequently, $J\in\mathcal G_n^n(A,y)$.  The level $p$ ancestor
of $J$ is one of the at most $10b$ intervals from $\mathcal D_p$ that
meet $B(x,R)$.  Applying
Lemma~\ref{Lem:phi-uniform-good-interval-estimate} to each of these
level $p$ intervals, with $N=n$, shows that the number of such intervals $J$ is at most
$$
10bC_0b^{(1-\theta)(n-p)},
$$
where $C_0$ is the constant that corresponds to the choice $A=B_\phi(1-a)^{-1}$ in Lemma~\ref{Lem:phi-uniform-good-interval-estimate}. These intervals cover $B(x,R)\cap \mathcal W_y$, and each of them has diameter
$b^{-n}\leq r$.  Moreover, by choice of $p,n$ we have
$$
b^{n-p}
<
\frac br R.
$$
It follows that
$$
N_r\bigl(B(x,R)\cap \mathcal W_y\bigr)
\leq
10b^{2-\theta}C_0
\left(\frac Rr\right)^{1-\theta}.
$$

It remains to consider $R\geq1$.  If $r\geq3/4$, then the two
intervals $[0,3/4], [1/4,1]$ cover $B(x,R)\cap \mathcal W_y$, so suppose $0<r<3/4$. Note that $\mathcal W_y$ is closed, and so there are $w_1, w_2\in \mathcal W_y$ such that 
$$
|w_1-w_2|=\diam \mathcal W_y\leq 1\leq R,
$$ which satisfy
$$
B(x,R)\cap \mathcal W_y\subset \mathcal W_y=(B(w_1,3/4)\cap \mathcal W_y)\bigcup (B(w_2,3/4)\cap \mathcal W_y).
$$ Therefore, the bound on the covering number from the previous case yields
$$
N_r\bigl(B(x,R)\cap \mathcal W_y\bigr)
\leq
20b^{2-\theta}C_0
\left(\frac {3/4}r\right)^{1-\theta}
\leq 20b^{2-\theta}C_0
\left(\frac Rr\right)^{1-\theta},
$$
which proves \eqref{eq:uniform-level-set-estimate} in the remaining case.

The fact that $\mathcal W_y$ is a porous subset of $\R$ follows directly by \cite[Theorem 5.2]{Luukkainen_porous}.
\end{proof}

\section{Assouad dimension of classical Weierstrass and Takagi functions}\label{sec: dimA cosine W}

In this section we  restrict our attention to classical Weierstrass and Takagi functions, and prove Theorems \ref{Thm:Aikawa-Weierstrass} and \ref{Thm: Takagi Conj} by employing Theorem \ref{Thm:Aikawa-Weierstrass QUANT}. 
We do so by establishing that the cosine and sawtooth generating functions satisfy appropriate $b$-adic (PD) conditions.
We then restrict to the parameter range $ab\geq \pi+1$ to establish appropriate $b$-adic (PD) conditions and the needed data in a constructive way. This yields explicit upper bounds for these cases by reproducing the proof of Theorem~\ref{Thm:Aikawa-Weierstrass QUANT} with explicit data.

\begin{proof}[Proof of Theorem \ref{Thm:Aikawa-Weierstrass}]
    We show that $\phi:\R\to \R$ with $\phi(t)=\cos(2\pi t)$ for all $t\in \R$ satisfies a $b$-adic (PD) condition with data $(n,1,0)$, where $n\in \N$ is a large integer depending only on $a,b$. Note that
    $$
    |P_n(b^{-n})-P_n(0)|=\sum_{j=0}^{n-1} a^j (1-\cos(2\pi b^j b^{-n})\geq a^{n-1}(1-\cos(2\pi b)).
    $$ Hence, it is enough to pick $n\in \N$ so that
    $$
    a^{n-1}(1-\cos(2\pi b))>\frac{\operatorname{Lip (\phi)}}{ab-1} \,b^{-n},
    $$ or, equivalently, so that
    $$
    (ab)^{n}>\frac{2\pi a}{(ab-1)(1-\cos(2\pi b))},
    $$ which is true for a sufficiently large $n\in \N$. The proof is then complete by an application of Theorem \ref{Thm:Aikawa-Weierstrass QUANT}.

\end{proof}

\begin{proof}[Proof of Theorem \ref{Thm: Takagi Conj}]
    Similarly to the proof of Theorem \ref{Thm:Aikawa-Weierstrass}, it can be shown that the generating function $\phi(t)=\dist(t, \Z):\R\to \R$ satisfies a $b$-adic (PD) condition with data $(n,1,0)$ for a sufficiently large $n$ that depends only on $a,b$. This is enough to complete the proof due to Theorem \ref{Thm:Aikawa-Weierstrass QUANT}.
\end{proof}

While the proofs of Theorems \ref{Thm:Aikawa-Weierstrass} and \ref{Thm: Takagi Conj} given above are based on the choice of some large integer $n$, we now include explicit upper bounds for $\dim_AG(W_{a,b})$ and $\dim_A G(T_{a,b})$ in the ranges $ab\geq \pi +1$ and $ab\geq 3$, respectively, by following a constructive argument. Note that the estimate in Theorem~\ref{Thm:Aikawa-Weierstrass QUANT}
depends on the data in the $b$-adic \eqref{eq:phi-finite-block-transversality}
condition and on the subsequent choices of $I_\phi$ and $L$. For the cosine
and sawtooth generating functions we make these choices  explicitly.
This yields the following bounds that depend solely on the parameters $a$ and
$b$.

\begin{Thm}
\label{Thm:explicit-cosine-Assouad-bound}
Suppose $a\in (0,1)$ and $b$ is an integer such that
$ab\geq\pi+1$. Then $\dim_A G(W_{a,b})$ is at most or equal to
$$
2-\frac{\log(13/12)}{\log b}\cdot
\left\lceil
\max\left\{
2,
\frac{\log\left(
\frac{\displaystyle
8\left(
ab-1+2\pi(1-a)
\right)}
{\displaystyle
(1-a)
\left(
(ab-1)
\sin\left(
\frac{\pi\lfloor b/3\rfloor}{b}
\right)
-
\frac{\pi\lfloor b/3\rfloor}{b}
\right)
}
\right)}{\log(1/a)}
\right\}
\right\rceil^{-1},
$$
with the above quantity  being strictly less than $2$.
\end{Thm}

\begin{proof}
Set
$$
\phi(t)=\cos(2\pi t),
\qquad
k_\phi=\lfloor b/3\rfloor, \qquad t_0=\frac14-\frac{k_\phi}{2b}.
$$
Since $a<1$ and $ab\geq\pi+1$, we have $b\geq5$. Moreover,
$$
\frac15\leq\frac{k_\phi}{b}\leq\frac13.
$$
Indeed, the upper bound is immediate, while
$\lfloor b/3\rfloor\geq(b-2)/3\geq b/5$ for $b\geq5$.
It follows that
$$
\frac1{12}\leq t_0\leq\frac3{20},
$$
and hence $t_0\in[0,1-k_\phi/b]$.

Recall that $P_1=\phi$ and
$$
\kappa_\phi=\frac{2\pi}{ab-1}.
$$
We have
\begin{align}
\left|
P_1\left(t_0+\frac{k_\phi}{b}\right)-P_1(t_0)
\right|
&=
\left|
\cos\left(\frac\pi2+\frac{\pi k_\phi}{b}\right)
-
\cos\left(\frac\pi2-\frac{\pi k_\phi}{b}\right)
\right| \nonumber \\
&=
2\sin\left(\frac{\pi k_\phi}{b}\right).\label{eq: explicit weier P}
\end{align}
Since $ab-1\geq\pi$, we have $\kappa_\phi\leq2$. Also, the
concavity of the sine function on $[0,\pi/2]$ gives
$$
\sin\left(\frac{\pi k_\phi}{b}\right)
\geq
\frac{2k_\phi}{b},
$$
which implies by $\kappa_\phi\leq2$ that
$$
2\sin\left(\frac{\pi k_\phi}{b}\right)
>
\frac{\kappa_\phi k_\phi}{b}.
$$
Thus, \eqref{eq: explicit weier P} and the above imply that $\phi$ satisfies a $b$-adic
\eqref{eq:phi-finite-block-transversality} condition with  data
$$
n_\phi=1,
\qquad
k_\phi=\lfloor b/3\rfloor,
\qquad
t_0=\frac14-\frac{\lfloor b/3\rfloor}{2b}.
$$

Following the notation in Section~\ref{sec: counting intervals}, set
$$
\eta
=
\frac14
\left(
2\sin\left(\frac{\pi k_\phi}{b}\right)
-
\frac{\kappa_\phi k_\phi}{b}
\right)
=
\frac12
\left(
\sin\left(\frac{\pi k_\phi}{b}\right)
-
\frac{\pi k_\phi}{b(ab-1)}
\right).
$$
We claim that the interval
$$
I_\phi
=
\left[
t_0-\frac1{13},
t_0+\frac1{13}
\right],
$$ is as in Section~\ref{sec: counting intervals}, satisfying \eqref{eq:phi-uniform-finite-block-transversality}.
Note that
$$
t_0-\frac1{13}
\geq
\frac1{12}-\frac1{13}
>0
$$
and
$$
t_0+\frac1{13}
\leq
\frac3{20}+\frac1{13}
<
\frac23
\leq
1-\frac{k_\phi}{b}.
$$
Thus, $I_\phi\subset(0,1-k_\phi/b)$. Moreover,
$$
|I_\phi|=\frac2{13}<\frac15\leq\frac{k_\phi}{b},
$$
so $I_\phi$ and $I_\phi+k_\phi/b$ are disjoint.

For every $t\in I_\phi$, the cosine difference formula gives
\begin{equation}\label{eq: explicit Weier PP}
    \left|
P_1\left(t+\frac{k_\phi}{b}\right)-P_1(t)
\right|
=
2\sin\left(\frac{\pi k_\phi}{b}\right)\cos(2\pi (t-t_0)).
\end{equation}
Since $|t-t_0|\leq1/13$, due to the elementary estimate
$\cos x\geq1-x^2/2$ we have
\begin{equation}\label{eq: explicit Weier cos}
    \cos(2\pi (t-t_0))
\geq
\cos(2\pi/13)
\geq
1-\frac{2\pi^2}{169}
>
\frac78.
\end{equation}
Moreover, the estimates $\kappa_\phi\leq2$ and
$\sin(\pi k_\phi/b)\geq2k_\phi/b$ imply
$$
\frac{\kappa_\phi k_\phi}{b}
\leq
\frac{2\sin\left(\frac{\pi k_\phi}{b}\right)}{2}.
$$
It follows by the above and \eqref{eq: explicit Weier PP}, \eqref{eq: explicit Weier cos} that
\begin{align*}
\left|
P_1\left(t+\frac{k_\phi}{b}\right)-P_1(t)
\right|
&\geq
\frac{7\left(2\sin\left(\frac{\pi k_\phi}{b}\right)\right)}{8}\\
&\geq
\frac{
3\left(2\sin\left(\frac{\pi k_\phi}{b}\right)\right)
+
\kappa_\phi k_\phi/b
}{4}\\
&=
\frac{\kappa_\phi k_\phi}{b}+3\eta,
\end{align*}
 for every $t\in I_\phi$.
Thus, $I_\phi$ satisfies
\eqref{eq:phi-uniform-finite-block-transversality}.

For the rest of the proof, denote by $L$ the integer
$$
\left\lceil
\max\left\{
2,
\frac{\displaystyle
\log\left(
\frac{\displaystyle
8\left(
ab-1+2\pi(1-a)
\right)}
{\displaystyle
(1-a)
\left(
(ab-1)
\sin\left(
\frac{\pi\lfloor b/3\rfloor}{b}
\right)
-
\frac{\pi\lfloor b/3\rfloor}{b}
\right)
}
\right)}
{\log(1/a)}
\right\}
\right\rceil.
$$
Since $b\geq5$, $L\geq2$, and $|I_\phi|=2/13$, we have
$$
b^L|I_\phi|
\geq
\frac{2b^2}{13}
>2.
$$
Furthermore, by $B_\phi=1$, $a\in (0,1)$, and the definition of $L$ we have
\begin{align*}
a^L
&\leq
\frac{
(1-a)
\left(
(ab-1)\sin\left(\frac{\pi k_\phi}{b}\right)
-
\frac{\pi k_\phi}{b}
\right)
}
{8\left(ab-1+2\pi(1-a)\right)}\\
&=
\frac{\displaystyle
\sin\left(\frac{\pi k_\phi}{b}\right)
-
\frac{\pi k_\phi}{b(ab-1)}
}
{\displaystyle
8\left(
\frac{1}{1-a}+\frac{2\pi}{ab-1}
\right)}.
\end{align*}
As a result,
\begin{align*}
a^L
\left(
\frac{B_\phi}{1-a}+\kappa_\phi
\right)
&\leq
\frac18
\left(
\sin\left(\frac{\pi k_\phi}{b}\right)
-
\frac{\pi k_\phi}{b(ab-1)}
\right)\\
&=
\frac{\eta}{4}.
\end{align*}
Therefore, $L$ satisfies \eqref{eq:phi-choice-of-L}. With these
explicit choices, \eqref{eq:phi-def-sigma-theta} becomes
$$
\theta_\phi
=
\frac{-\log(1-1/13)}{L\log b}
=
\frac{\log(13/12)}{L\log b}.
$$
By Theorem~\ref{Thm:Aikawa-Weierstrass QUANT} we have
$$
\dim_A G(W_{a,b})
\leq
2-\frac{\log(13/12)}{L\log b},
$$
and substituting the exact choice of $L$  yields
the desired inequality. 
\end{proof}

Similarly, using Theorem~\ref{Thm:Aikawa-Weierstrass QUANT} we obtain an explicit upper bound for the corresponding Takagi functions. Since the proof is almost identical to that of Theorem~\ref{Thm:explicit-cosine-Assouad-bound}, we mostly focus on the appropriate choice of data for the (PD) condition and the interval $I_\phi$.

\begin{Thm}
\label{Thm:explicit-Takagi-Assouad-bound}
Suppose $a\in (0,1)$ and $b>1$ is an integer such that $ab\geq3$. Then
$$
\dim_A G(T_{a,b})
\leq
2-\frac{\log(25/23)}{\log b}\cdot 
\left\lceil
\max\left\{
2,
\frac{\displaystyle
\log\left(
\frac{\displaystyle
8b(ab+1-2a)}
{\displaystyle
(1-a)\lfloor b/3\rfloor(ab-2)}
\right)}
{\log(1/a)}
\right\}
\right\rceil^{-1},
$$
with the quantity on the right-hand side being strictly less than $2$.
\end{Thm}

\begin{proof}
Set $\phi(t)=\dist(t,\Z)$. Following similar arguments as in the proof of
Theorem~\ref{Thm:explicit-cosine-Assouad-bound}, it can be shown that
$\phi$ satisfies a $b$-adic
\eqref{eq:phi-finite-block-transversality} condition with data
$$
n_\phi=1,
\qquad
k_\phi=\lfloor b/3\rfloor,
\qquad
t_0=0.
$$ The corresponding integral can be chosen to be
$$
I_\phi
=
\left[
\frac1{150},
\frac16
\right],
$$ and the corresponding integer $L$ can be chosen to be
$$
\left\lceil
\max\left\{
2,
\frac{\displaystyle
\log\left(
\frac{\displaystyle
8b(ab+1-2a)}
{\displaystyle
(1-a)\lfloor b/3\rfloor(ab-2)}
\right)}
{\log(1/a)}
\right\}
\right\rceil.
$$

\end{proof}

For a Lipschitz $1$-periodic function $\phi$, we use henceforth the notation
$$
\operatorname{osc}\phi
=
\operatorname{osc}_{[0,1]}\phi.
$$
We finish this section by recording a sufficient condition for a given $\phi$ to satisfy a $b$-adic (PD) condition. This involves solely the Lipschitz constant and the oscillation of $\phi$, and is thus easier to check for a given generating function $\phi$.

\begin{Prop}
\label{Prop:phi-oscillation-implies-PD}
Suppose that $a\in (0,1)$ and $b\in\N$ satisfy $ab>1$, and let
$\phi:\R\to\R$ be Lipschitz and $1$-periodic. If
\begin{equation}
\label{eq:phi-oscillation-PD-condition}
\frac{\operatorname{Lip (\phi)}}{ab-1}
<
\frac{2\operatorname{osc}\phi}{1+a},
\end{equation}
then $\phi$ satisfies a $b$-adic (PD) condition.
\end{Prop}

\begin{proof}
Recall that
$$
\kappa_\phi= \frac{\operatorname{Lip (\phi)}}{ab-1}.
$$The strict inequality in
\eqref{eq:phi-oscillation-PD-condition} implies that
$\operatorname{osc}\phi>0$. Thus, $\phi$ is nonconstant,
$\operatorname{Lip}(\phi)>0$ and $\kappa_\phi>0$.

Assume towards contradiction that $\phi$ does not satisfy any
$b$-adic positive difference condition. Then, for every integer
$n\geq1$, every integer $1\leq k<b^n$, and every
$t\in[0,1-kb^{-n}]$, we have
$$
\left|P_n(t+kb^{-n})-P_n(t)\right|
\leq
\kappa_\phi kb^{-n}.
$$
Fix $n\in \N$, and
$0\leq x_1\leq x_2\leq1$ with $x_2-x_1<1$. Suppose $\left\lfloor b^n(x_2-x_1)\right\rfloor\geq 1$. Since $\left\lfloor b^n(x_2-x_1)\right\rfloor< b^n$ and $x_1\in [0,1-\left\lfloor b^n(x_2-x_1)\right\rfloor b^{-n}]$, by applying the above inequality we have
\begin{align*}
|P_n(x_2)-P_n(x_1)|
&\leq
\left|P_n(x_2)
-
P_n\left(x_1+\left\lfloor b^n(x_2-x_1)\right\rfloor b^{-n}\right)\right|\\
&\quad+
\left|P_n\left(x_1+\left\lfloor b^n(x_2-x_1)\right\rfloor b^{-n}\right)
-
P_n(x_1)\right|\\
&\leq
\operatorname{Lip}(P_n)
\left(x_2-x_1-\left\lfloor b^n(x_2-x_1)\right\rfloor b^{-n}\right)
+
\kappa_\phi\left\lfloor b^n(x_2-x_1)\right\rfloor b^{-n}.
\end{align*}
If $\left\lfloor b^n(x_2-x_1)\right\rfloor=0$, the second difference
in the first inequality is zero, so the same estimate follows
directly from the Lipschitz continuity of $P_n$. Therefore, in either case, by definition of the floor function the above implies
\begin{equation}\label{eq: osc Lip PD upper bound}
    |P_n(x_2)-P_n(x_1)|\leq
\kappa_\phi(x_2-x_1)
+
\operatorname{Lip}(P_n)b^{-n}.
\end{equation} Note that $n, x_1, x_2$ in the above estimate are arbitrary.

If $P_n$ is constant, then
$$
\operatorname{osc}P_n
\leq
\frac{\kappa_\phi}{2}
+
\frac{3\operatorname{Lip}(P_n)}{2b^n}
$$
holds trivially. Otherwise, since $P_n$ is continuous and
$1$-periodic, there are $x,y\in[0,1]$ with
$0\leq x<y\leq1$ and $0<y-x<1$ such that
$$
|P_n(y)-P_n(x)|
=
\operatorname{osc}P_n.
$$
Applying \eqref{eq: osc Lip PD upper bound} for $x_1=x$, and  $x_2=y$ gives
\begin{equation}\label{eq: adding oscPn propos PD 1}
    \operatorname{osc}P_n
\leq
\kappa_\phi(y-x)
+
\operatorname{Lip}(P_n)b^{-n}.
\end{equation}
On the other hand, using $P_n(0)=P_n(1)$ and applying \eqref{eq: osc Lip PD upper bound} separately for $x_1=y,x_2=1$ and for $x_1=0, x_2=x$, we obtain
\begin{align*}
\operatorname{osc}P_n
&=
|P_n(y)-P_n(x)|\\
&\leq
|P_n(y)-P_n(1)|
+
|P_n(0)-P_n(x)|\\
&\leq
\kappa_\phi(1-y+x)
+
2\operatorname{Lip}(P_n)b^{-n}.
\end{align*}
Adding \eqref{eq: adding oscPn propos PD 1} with the above yields
\begin{equation}\label{eq: eq: adding oscPn propos PD 2}
    \operatorname{osc}P_n
\leq
\frac{\kappa_\phi}{2}
+
\frac{3\operatorname{Lip}(P_n)}{2b^n}.
\end{equation}
Moreover, recall that
\begin{align*}
\operatorname{Lip}(P_n)
&\leq
\operatorname{Lip}(\phi)
\sum_{j=0}^{n-1}(ab)^j\\
&=
\kappa_\phi\big((ab)^n-1\big),
\end{align*}
and hence
$$
\frac{\operatorname{Lip}(P_n)}{b^n}
\leq
\kappa_\phi\left(a^n-b^{-n}\right).
$$
It follows by \eqref{eq: eq: adding oscPn propos PD 2} and the above that
\begin{equation}
\label{eq:upper-oscillation-Pn-PD}
\operatorname{osc}P_n
\leq
\frac{\kappa_\phi}{2}
+
\frac{3\kappa_\phi}{2}\left(a^n-b^{-n}\right).
\end{equation}

For every $t\in\R$, we also have
$$
\phi(t)
=
P_n(t)-aP_n(bt)+a^n\phi(b^nt).
$$
Since $b\in\N$ and both $\phi$ and $P_n$ are $1$-periodic, by the definition of
oscillation and triangle inequality the above gives
$$
\operatorname{osc}\phi
\leq
(1+a)\operatorname{osc}P_n
+
a^n\operatorname{osc}\phi,
$$
which implies
$$
\frac{1-a^n}{1+a}\operatorname{osc}\phi
\leq
\operatorname{osc}P_n.
$$
Combining this with
\eqref{eq:upper-oscillation-Pn-PD}, we obtain
$$
\frac{1-a^n}{1+a}\operatorname{osc}\phi
\leq
\frac{\kappa_\phi}{2}
+
\frac{3\kappa_\phi}{2}\left(a^n-b^{-n}\right).
$$
Since $n\in \N$ is arbitrary, letting $n\to\infty$ gives
$$
2\operatorname{osc}\phi
\leq
\kappa_\phi(1+a),
$$
contradicting \eqref{eq:phi-oscillation-PD-condition}. Therefore
$\phi$ satisfies a $b$-adic (PD) condition.
\end{proof}

\begin{Rem}
    Theorems \ref{Thm:Aikawa-Weierstrass} and \ref{Thm: Takagi Conj} also follow by an application of Proposition \ref{Prop:phi-oscillation-implies-PD}, although for the slightly restricted ranges $ab\geq \pi+1$ and $ab\geq 3$.
\end{Rem}


\section{Final remarks}\label{sec: Final Rem}

A very natural question following Theorem~\ref{Thm:Aikawa-phi-Weierstrass} is whether the (PD) condition assumption can be improved. Namely, it is clear that this assumption is not a characterization of the generating functions $\phi$ that yield Weierstrass type functions whose graph has Assouad dimension less than $2$. An almost trivial example would be to choose any constant generating function $\phi$, which would not satisfy any (PD) condition, and yet the corresponding Weierstrass function has a graph of Assouad dimension $1$. It would be interesting to establish a clear dichotomy on generating functions $\phi$ that yield to Weierstrass functions with graph of Assouad dimension either $2$ or less than $2$, perhaps in the spirit of the result of Ren-Shen \cite{RenShenInventWeier}.

Recall that the formula for the Assouad dimension of $G(T_{a,b})$ proved in \cite{TakagiDimA} is for the case $a\in (1/2,1)$ and $b=2$, which is a special case of $ab<2$. On the other hand, Theorem~\ref{Thm: Takagi Conj} shows that $\dim_A G(T_{a,b})<2$ for $a\in (0,1)$, $b\in \N$ with $ab>1$. Can the formula from \cite{TakagiDimA} be proved for the general case of Theorem~\ref{Thm: Takagi Conj}? Furthermore, Yu's result in \cite{Yu_Takagi} shows that the Assouad dimension of $G(T_{a,b})$ is strictly greater than the upper box dimension, under certain assumptions on $a\in (0,1)$ and $b\in \N$, which include the condition $ab\leq 2$. Can it be shown that $\dim_A G(T_{a,b})$ is strictly greater than the upper box dimension of $G(T_{a,b})$ for all $a\in (0,1)$, $b\in \N$ with $ab>1$?

While Theorem \ref{Thm:Aikawa-Weierstrass QUANT} establishes a quantitative upper bound on $\dim_A G(W_{a,b}^\phi)$ away from $2$, this bound is unlikely to be sharp. This can perhaps already be seen by Theorem~\ref{Thm:explicit-cosine-Assouad-bound} in the classical Weierstrass functions case, and similarly for the Takagi functions and Theorem~\ref{Thm:explicit-Takagi-Assouad-bound}. It would be interesting to achieve a sharper upper bound than $2-\theta_\phi$, at least in the case of the classical Weierstrass function with $\phi(x)=\cos(2\pi x)$. Such an approach would need to employ explicit properties of the cosine function, for instance by improving the upper bound in Lemma~\ref{Lem:phi-L-level-uniform-good-interval-estimate} directly. Such an improved bound would afterwards result in sharper estimates by following exactly the same strategy as in Sections \ref{sec: counting intervals} and \ref{sec: dimA general W}. It should  be noted that attention regarding these questions is restricted to the non-smooth parameter range $ab>1$, including as stated in \cite{FraserBook}, since for $0<ab\leq 1$ the fractal dimension notions of the graph are all equal to $1$ (see for instance \cite{SmoothTakagi}).


One of the most important implications of Theorem \ref{Thm:Aikawa-phi-Weierstrass} is the porosity property that the graph is proved to carry in Corollary \ref{cor: graph porous} due to the work of Luukkainen \cite{Luukkainen_porous}. While this is a geometric property in nature, it leads to various analytic applications for the graph of Weierstrass functions. We devote the rest of the section to such applications that follow directly from Theorem \ref{Thm:Aikawa-Weierstrass QUANT}, Corollary \ref{cor: graph porous} and the work of other authors on Muckenhoupt weights and Hardy inequalities. 

A particularly useful application is on distance weights used in harmonic analysis. In particular, porosity notions have been very closely connected to the Muckenhoupt properties of such weights. We refer to \cite{Anderson_Mudarra_Lehr_Vah} and the references therein for more details on the importance of such weights. Given $a\in (0,1)$ and an integer $b> 1/a$, set 
$\phi(t)=\cos(2\pi t)$ and $G=G(W^\phi_{a,b})$ 
for the rest of the section to restrict our attention to classical Weierstrass functions, although the corresponding results hold for generalized Weierstrass functions under the assumptions of Theorem \ref{Thm:Aikawa-phi-Weierstrass}. For any $p\in (1,\infty)$, there is some $\alpha>0$ with 
$$
(1-p)(2-\dim_A G)<\alpha<2-\dim_A G
$$ such that the weight $\dist(z,G)^{-\alpha}$ lies in the Muckenhoupt class $A_p$. In particular, due to a result of Dyda, Ihnatsyeva, Lehrb\"ack, Tuominen, and V\"ah\"akangas we have the following characterization.

\begin{Cor}
\label{cor:Muckenhoupt-distance-weights}
Let $1<p<\infty$, $\alpha\in\mathbb R$, and set $w=\dist(z,G)^{-\alpha}$. Then we have $w\in A_p$, i.e.,
$$
 \sup_B
\left(\fint_Bw\,dz\right)
\left(\fint_Bw^{-1/(p-1)}\,dz\right)^{p-1}
<\infty,
$$ if, and only if,
$$
(1-p)(2-\dim_A G)<\alpha<2-\dim_A G.
$$
\end{Cor}

\begin{proof}
The proof follows by the definition of  Muckenhoupt weights $A_p$, Corollary~3.8 of
\cite{DydaIhnatsyevaLehrbackTuominenVahakangas},  and Theorem \ref{Thm:Aikawa-Weierstrass}.
\end{proof}
It would be interesting to attempt and sharpen the upper bound in Theorem \ref{Thm:Aikawa-Weierstrass QUANT} by proving Muckenhoupt properties for appropriate exponents of the distance function, as in the above characterization. An additional analytic characterization of the Assouad (co-)dimension of $G$ is in terms of Triebel-Lizorkin Hardy inequalities studied by Ihnatsyeva and V\"ah\"akangas
\cite{Ihnatsyeva_Vahakangas_Hardy_II}. It would be interesting to determine whether sharper bounds can be achieved for $\dim_A G$ through this route of Hardy inequalities. For the exact definition of the
Triebel--Lizorkin space $F^s_{p,q}(\R^2)$ and the corresponding  norm $\lVert f\rVert_{F^s_{p,q}(\R^2)}$ we refer to \cite{Ihnatsyeva_Vahakangas_Hardy_II} and the references therein.

\begin{Cor}
\label{cor:Triebel-Lizorkin-Hardy-Weierstrass}
Let $1<p<\infty$, $1\leq q<\infty$, and $0<s<2/p$.  The following
statements are equivalent:
\begin{enumerate}
\item $\dim_A G<2-sp$;
\item there is $C>0$ such that
$$
\left(
\int_{\R^2\setminus G}
\frac{|f(z)|^p}{\dist(z,G)^{sp}}\,dz
\right)^{1/p}
\leq
C\lVert f\rVert_{F^s_{p,q}(\R^2)}
$$
for every $f\in F^s_{p,q}(\R^2)$.
\end{enumerate}
In particular, (2) holds for all $sp<\theta_\theta$, where $\theta_\phi$ is as in \eqref{eq:phi-def-sigma-theta} for $\phi(t)=\cos(2\pi t)$.
\end{Cor}

\begin{proof}
    The proof follows by an application of \cite[Theorem 1.2]{Ihnatsyeva_Vahakangas_Hardy_II} on the domain $\R^2 \setminus G$ and the estimate in Theorem \ref{Thm:Aikawa-Weierstrass QUANT}.
\end{proof}

\end{document}